\documentclass[12pt,a4paper]{amsart}

\makeatletter
\tagsleft@false
\makeatother

\usepackage{mathrsfs}
\usepackage{tikz}
\usepackage{amsfonts}
\usepackage{amsmath}
\usepackage{amssymb}
\usepackage{enumerate}
\usepackage{hyperref}
\usepackage{latexsym}
\usepackage{xcolor}
\usepackage[shortlabels]{enumitem}
\usepackage[numbers,sort&compress]
{natbib}
\setlist[enumerate,1]{label={\upshape(\roman*)}}

 \makeatletter
    
    \newcommand{\Rmnum}[1]
    {\expandafter\@slowromancap\romannumeral #1@}
    \makeatother

\newtheorem{thm}{Theorem}[section]

\newtheorem{lemma}[thm]{Lemma}

\newtheorem{cor}[thm]{Corollary}

\newtheorem{example}[thm]{Example}
\newtheorem{defin}[thm]{Definition}

\theoremstyle{definition}
\newtheorem{remark}[thm]{Remark}

\title[Bound on shortest cycle covers]{Bound on shortest cycle covers}

\date{}

\author[Song]{Deping Song}
\address{School of Mathematical Sciences\\Zhejiang Normal University\\Jinhua, Zhejiang 321004\\China}
\email{songdeping@zjnu.edu.cn }

\author[Zhu]{Xuding Zhu}
\address{School of Mathematical Sciences\\Zhejiang Normal University\\Jinhua, Zhejiang 321004\\China}
\email{xdzhu@zjnu.edu.cn }

\begin{document}

\begin{abstract}
Assume $G$ is a bridgeless graph. A cycle cover of $G$ is a collection of cycles of $G$ such that each edge of $G$ is contained in at least one of the cycles. The length of a cycle cover of $G$ is the sum of the lengths of the cycles in the cover. The minimum length of a cycle cover of $G$ is denoted by $cc(G)$. It was proved independently by 
Alon and Tarsi   and by Bermond, Jackson, and Jaeger  that $cc(G)\le \frac{5}{3}m$ for every bridgeless graph $G$ with $m$ edges. This remained the best-known upper bound for $cc(G)$ for 40 years.
In this paper, we prove that if $G$ is a bridgeless graph with $m$ edges and $n_2$ vertices of degree $2$, then $cc(G) < \frac{29}{18}m+ \frac 1{18}n_2$. As a consequence, we show that $cc(G) \le  \frac 53 m - \frac 1{42} \log m$. 
The upper bound
$ cc(G) < \frac{29}{18}m \approx  1.6111 m$ for bridgeless graphs $G$ of minimum degree at least 3 improves the previous known upper bound $1.6258m$. A key lemma used in the proof confirms Fan's conjecture that if $C$ is a circuit of $G$ and $G/C$ admits a  nowhere zero 4-flow, then $G$ admits a 4-flow $f$ such that $E(G)-E(C)\subseteq \text{supp} (f)$ and $|\textrm{supp}(f)\cap E(C)|>\frac{3}{4}|E(C)|$.
\end{abstract}

\keywords{}

\maketitle
\section{Introduction}\label{sec1}

Graphs considered in this paper may have loops and multiple edges. For terminology and notations not defined here, we follow~\cite{Bondy,Zhang}. {\it Contracting an edge } means deleting the edge and then identifying its ends. For a subgraph $H$ of a graph $G$, let $G/H$ be the graph obtained from $G$ by contracting all edges of $H$.

A {\it circuit} is a $2$-regular connected graph, and a {\it cycle} is a  graph such that the degree of each vertex is even. The {\it length} of a cycle or a circuit is the number of its edges. A collection $\mathcal{C}$ of cycles of a graph $G$ {\it covers}  $G$ if each edge of $G$ is contained in at least one cycle in $\mathcal{C}$.  The {\it length} of a cycle cover $\mathcal{C}$, denoted by $\ell(\mathcal{C})$, is the sum of the lengths of its cycles. The length of the shortest cycle cover of $G$ is denoted by $cc(G)$.

Alon and Tarsi \cite{Alon} conjectured that $cc(G)\le \frac{7}{5}m$ for every bridgeless graph $G$ with $m$ edges.
Jamshy and Tarsi \cite{Jamshy} proved that Alon and Tarsi's conjecture  implies the well-known cycle double cover conjecture of Seymour \cite{Seymour} and Szekeres \cite{Szekeres},
 that every bridgeless graph has a collection of cycles that covers each edge exactly twice. 

 Alon and Tarsi \cite{Alon} and Bermond, Jackson, and Jaeger \cite{Bermond} independently proved that $cc(G)\le \frac{5}{3}m$ for every bridgeless graph $G$ with $m$ edges. 
 This upper bound on $cc(G)$ remained the best known upper bound for 40 years. 
 For special families of graphs, there are some better upper bounds proved in the literature. 
  Jamshy, Raspaud,
and Tarsi \cite{JRT} proved that
if $G$ admits a nowhere-zero 5-flow, then $cc(G) \le \frac 85 m$.   Fan and Raspaud \cite{FR}  proved  that if the Fulkerson Conjecture is true, then $cc(G) \le 
\frac {22}{15}m$.   Kaiser, Kr\'{a}l, Lidick\'{y}, Nejedl\'{y} and \v{S}\'{a}mal  \cite{K} proved that if $G$ has minimum degree at least 3 and is loopless, then $cc(G)\le \frac{44}{27}m$.   Fan \cite{Fan} proved that if $G$ has minimum degree at least 3, then $cc(G)<\frac{218}{135}m$ in the loopless case and $cc(G)<\frac{278}{171}m$ if loops are allowed.
 Kompi\v{s}ov\'{a}     and Lukot'ka \cite{Lu7} proved that $cc(G)< \frac{278}{171}m+\frac{2}{27}n_2$ for an arbitrary bridgeless graph $G$, where $n_2$ is the number of degree 2 vertices in $G$.

This paper proves that $cc(G)<  \frac{29}{18}m+ \frac{1}{18}n_2$.
As a consequence, we show that $cc(G)\le \frac 53 m - \frac 1{42} \log m$, which improves the general upper bound for $cc(G)$ by a fraction of $\log m$. 

Let $G$ be a graph and $\Gamma$ be an (additive) abelian group.  For an orientation $D$  of $E(G)$, denote $E^+(v)$ (resp., $E^-(v)$) by the set of all oriented edges of $D(G)$ with their tails (resp., heads) at   vertex $v\in V(G)$. For a function $f: E(G)\to \Gamma$ and $v\in V(G)$, let $f^+(v)=\sum_{e\in E^+(v)}f(e)$ and $f^-(v)=\sum_{e\in E^-(v)}f(e)$. A {\it $\Gamma$-flow} in $G$ is an ordered pair $(D,f)$ such that $f^+(v)-f^-(v)=0$ for every vertex $v\in V(G)$. 
For a subgraph $H$ of $G$ and $a\in \Gamma$,  define  $E_{f=a}(H)=\left\lbrace e\in E(H):f(e)=a \right\rbrace $. 
The {\it support} of $f$ is the set of all edges of $G$ with $f(e)\ne 0$ and is denoted by $supp(f)$. A flow $(D,f)$  is called {\it nowhere-zero} if $supp(f)=E(G)$.


We shall consider $\mathbb{Z}_2\times\mathbb{Z}_2$-flows in this paper.
As $-\alpha = 
\alpha$ for all $\alpha \in \mathbb{Z}_2\times\mathbb{Z}_2$, the orientation of $G$ is irrelevant and can be omitted. To be precise, a $Z_2 \times Z_2$-flow on $G$ is a function $f: E(G) \to Z_2 \times Z_2$ such that for each vertex $v$, $\sum_{e \in E(v)}f(e) = 0$, where $E(v)$ is the set of edges of $G$ incident to $v$.

\section{A key Lemma}

A maximal path $P$ in a graph $H$ with all interior vertices of degree $2$ (in $H$) is called a {\em thread} in $H$.
This section proves the following lemma.

\begin{lemma} \label{main1}
	Assume $G$ is a graph and $C$ is a circuit in $G$. 
    Assume each edge in $C$ is contained in a thread of $G$ of length at least $q$.
    If $G$ admits a $\mathbb{Z}_2 \times \mathbb{Z}_2$-flow $f$, then $G$ admits a $\mathbb{Z}_2 \times \mathbb{Z}_2$-flow $g$ such that ${\rm supp}(f) - C = {\rm supp}(g) -C$ and  $ |E_{g=(0,0)}(C)|\le \frac{1}{4}(|E(C)|-q)$. 
\end{lemma}
\begin{proof}
Assume $f$ is a $\mathbb{Z}_2 \times \mathbb{Z}_2$-flow on $G$ and $C$ is a circuit. A subpath $S$ of $C$ is an $ f_{(0,0),(1,1)}$-\it segment \rm in $C$ if $f(e) \in \{(0,0),(1,1)\}$ for each edge $e \in S$. We choose $f$ and $S$ in such a way that for any $\mathbb{Z}_2\times \mathbb{Z}_2$-flow $g$ with ${\rm supp}(f) - C = {\rm supp}(g) -C$, and any $g_{(0,0), (1,1)}$-segment $S'$ in $C$, $|S'| \le |S|$. 

For $a\in \mathbb{Z}_2\times \mathbb{Z}_2$ and cycle $C'$, let $f_{C', a}$ be the function defined as $f_{C',a}(e)=a$ if $e \in E(C')$, and $f_{C',a}(e) = 0$ for other edges $e$.  Then $f_{C',a}$ is a $\mathbb{Z}_2\times \mathbb{Z}_2$-flow with ${\rm supp}(f_{C',a}) = E(C')$. 

If $S=C$, then let $g=f+f_{C,(0,1)}$. Then $g$ is a  $Z_2\times\mathbb{Z}_2$-flow 
on $G$ with ${\rm supp}(f) - C = {\rm supp}(g) -C$, and $E_{g=(0,0)}(C) = \emptyset$. Thus, we may assume that $S\ne C$.

	Assume $S=v_1v_2\cdots v_{s+1}$, and $C=(v_1v_2\ldots v_p)$ (so the other neighbor of $v_1$ in $C$ is $v_p$). Considering $f$ or $f+f_{C,(1,1)}$, we may suppose $f(v_1v_2)=(0,0)$.

	By the choice of $S$, we have $f(v_pv_1)\notin \left\lbrace (0,0),(1,1)\right\rbrace $. By symmetry, we may assume that $f(v_pv_1)=(1,0) $.

    It is well known and easy to verify that for any 2-subset $\{a,b\}$ of $\{(1,0), (0,1), (1,1)\}$, $E_{f=a}(G)\cup E_{f=b}(G)$ is a cycle. Let $C'=E_{f=(1,0)}(G)\cup E_{f=(1,1)}(G)$ and $A=C'-C$. Then $C'$ is a cycle and hence the set $O$ of all vertices of odd degree in $A$ is a subset of $V(C)$ and $v_1\in O$. Thus, there is a vertex $v_t\in O-\left\lbrace v_1\right\rbrace $ in the same component of $A$ as $v_1$. Let $P$ be the path in $A$ from $v_1$ to $v_t$, and $P'$ be the subpath of $C$ from $v_1$ to $v_t$ containing $v_1v_2$. Let $C_1$ be the cycle consisting of $P$ and $P'$, and $C_2=C\Delta C_1$. Let $g=f+f_{C_1,(0,1)}$, $h=f+f_{C_2,(0,1)}$. Then $g,h$ are $\mathbb{Z}_2\times \mathbb{Z}_2$ flows with ${\rm supp}(f) - C = {\rm supp}(g) -C = {\rm supp}(h) -C$. 
	
	If $t\ge s+1$, then $h(e)=f(e)$ for $e\in E(S)$ and $h(v_pv_1)=(1,1)$, and therefore $S'=(v_pv_1v_2\ldots v_{s+1})$ is a $h_{(0,0),(1,1)}$-segment in $C$ with $|S'| > |S|$, in contrary to our choice of $f$ and $S$.

	Assume $t\le s$. Then ${\rm supp}(f) \subseteq {\rm supp}(g)$. Note that the edge $v_1v_2$ is contained in a thread $P$ of length at least $q$. All the edges $e$ in $P$ have $f(e)=(0,0)$ and $g(e) = (0,1)$. Hence 
     $|E_{g=(0,0)}(C)| \le |E_{f=(0,0)}(C)| -q$.
     If $|E_{g=(0,0)}(C)| \le \frac 14 (|C|-q)$, then we are done. Otherwise $|E_{f=(0,0)}(G)| \ge \frac 14 |C|+ \frac 34 q$.
Thus, there exists $b\in \mathbb{Z}_2\times \mathbb{Z}_2 - \{(0,0)\}$ such that $|E_{f=b}(C)| \le \frac 14 (|C|-q)$. Let  $g=f+f_{C,b}$, then $|E_{g=(0,0)}(C)|\le \frac{1}{4}(|E(C)|-q)$.
	\end{proof}
%

For a positive integer $k$, a   {\it $k$-flow}   is a $\mathbb{Z}$-flow $f$ such that $|f(e)|< k$ for each edge $e$.  Tutte \cite{T} proved that
a graph $G$ admits a flow nowhere-zero $k$ if and only if $G$ admits a flow nowhere-zero $\Gamma$-flow for any abelian group $\Gamma$ of order $k$. Thus the following is a corollary of Lemma \ref{main1}.

\begin{cor} \label{main}
	Assume $G$ is a graph and $C$ is a circuit in $G$. If $G/C$ admits a nowhere-zero $4$-flow, then $G$ admits a $4$-flow $f$ such that $E(G)-E(C)\subseteq supp(f)$ and $ |E_{f=0}(C)|<\frac{1}{4}|E(C)|$.
\end{cor}

Corollary \ref{main} was conjectured by Fan \cite{Fan}, where it was confirmed for circuits $C$ with $|C| \le 19$, and  was further confirmed  for $|C| \le 27$ and for $|C| \le 35$  in \cite{WLG} and \cite{W}, respectively.

\section{Shortest cycle cover of brideless graphs}

This section proves the following bound of shortest cycle cover of general graphs. 

\begin{thm} \label{cc1}
	Let $G$ be a bridgeless graph with $m$ edges and $n_2$ vertices of degree $2$, then $cc(G)< \frac{29}{18}m+\frac{1}{18}n_2$ and consequently $cc(G)\le  \frac 53 m - \frac 1{42} \log m$, where  the base for $\log$ is 2. Moreover, if $m \ge 141$, then $cc(G) \le \frac 53 m - \frac {25}3$. 
\end{thm}
\begin{proof}
We may assume that $G$ is loopless: if $e$ is a loop in $G$, then $G'=G-e$ has at most one more degree 2 vertex, and $cc(G) = cc(G')+1$. If Theorem \ref{cc1} holds for $G'$, then it holds for $G$ as well.

Let $\omega(e)=1$ for each edge in $G$. If there is a vertex $v$ of degree at least $4$, then by a result in \cite{Fleischner}, we can create a new bridgless graph $G'$ by splitting $v$ into   $v'$ and $v''$ so that $v''$ is adjacent to two neighbours of $v$ and $v'$ is adjacent to the remaining neighbours of $v$. 
We add an edge of weight $0$ that connects   $v'$ and $v''$.   Repetition of this operation results in an edge weighted graph $H$ whose maximum degree is at most $3$. Moreover, when all  edges of weight $0$ are removed, the resulting graph is bridgeless. 

 Let $\bar{H}$ be the graph obtained from $H$ by replacing each thread $P$ with a single edge $e_P$. Thus $\bar{H}$ is a bridgeless cubic graph. By \cite{ED}, there is an integer $p\in \mathbb{Z}^+$ and a family $\bar{\mathcal{M}}$ of $3p$ perfect matchings such that every edge $e\in E(\bar{H})$ is contained in $p$ members of $\mathcal{P}$.
For a perfect matching $\bar{M} \in \bar{\mathcal{M}}$, $\bar{F}=\bar{H}-\bar{M}$ is a cycle in $\bar{H}$.

For each edge $e_P$ in $\bar{H}$, replace $e_P$ by the path $P$, and contract all edges $e$ of weight $0$.  Each $\bar{F}$ corresponds to a cycle $F$ in $G$.
We denote by $\mathcal{F}$ the family cycles $F$ obtained from $\bar{F}$, and denote by $\mathcal{M}$ the family of subgraphs $M$ obtained from $\bar{M}$. Then  each edge $e$ of $G$ is contained in $2p$ cycles $F$ in $\mathcal{F}$, and contained in $p$ subgraphs in $\mathcal{M}$. The average size of $F$ is $2m/3$ and the average size of $M$ is $m/3$.

Let $F$ be a cycle in $\mathcal{F}$ of maximum size. Thus $|F| \ge \frac 23 m$.
It was shown in \cite{Zhang} that $G/F$ has odd-edge-connectivity at least $5$.
Hence $G/F$ have a nowhere-zero $\mathbb{Z}_2 \times \mathbb{Z}_2$-flow \cite{J}, which extends to a $\mathbb{Z}_2 \times \mathbb{Z}_2$-flow $f$ of $G$ with $E(G)-F \subseteq {\rm supp}(f)$.

By Lemma \ref{main1}, we may assume that for each component $B$ of $F$, 
$|E_{f=(0,0)}, B)| < \frac 14 |B|$. 
For $i \ge 2$, let $d_i$ be the number of components in $F$ with $i$ edges. 
Then  $$|E_{f=(0,0)}(G)|=|E_{f=(0,0)}(F)|\le \frac{1}{4}|F|-\frac{1}{4}\sum\limits_{i=1}^{\infty}d_{4i+1}-\frac{1}{2}\sum\limits_{i=0}^{\infty}d_{4i+2}-\frac{3}{4}\sum\limits_{i=0}^{\infty}d_{4i+3}-\sum\limits_{i=1}^{\infty}d_{4i}.$$ 

We choose the $\mathbb{Z}_2 \times \mathbb{Z}_2$-flow $f$ so that $|E_{f=(0,1)}(G-F)|$ is minimum, and subject to this, $|E_{f=(1,0)}(G-F)|$ is minimum.
This implies that $E_{f=(0,1)}(G-F)$ is acyclic, because if $C$ is a cycle in $E_{f=(0,1)}(G-F)$, then $f+f_{C, (1,0)}$ is a $\mathbb{Z}_2 \times \mathbb{Z}_2$-flow $g$ with ${\rm supp}(g)={\rm supp}(f)$ and $|E_{g=(0,1)}(G)| < | E_{f=(0,1)}(G)|$. Similarly,
$E_{f=(1,0)}(G-F)$ is acyclic. 

 Let $C_1=E_{f=(0,1)}(G) \cup E_{f=(1,0)}(G),C_2=E_{f=(0,1)}(G) \cup E_{f=(1,1)}(G),C_3=E_{f=(1,1)}(G) \cup E_{f=(1,0)}(G)$. Then $C_1, C_2,C_3$ is a double cover of $supp(f)$. Therefore $$\mathcal{C}_1= \{C_1\Delta F,C_2\Delta F,C_3\Delta F\}$$ is a cycle cover of $G$, where each edge not in $F$ is covered twice, and each edge $e \in F \cap {\rm supp}(f)$ is covered once  and each edge in $E_{f=0}(F)$ is covered three times. Thus $\mathcal{C}_1$ is a cycle cover of $G$ of 
length $2m-|F|+2|E_{f=0}(G)|$. Therefore 
\begin{eqnarray*}
    \ell(\mathcal{C}_1) &\le& 2m-|F|+2|E_{f=0}(G)| \\
    &\le& 2m-\frac{1}{2}|F|-\frac{1}{2}\sum\limits_{i=1}^{\infty}d_{4i+1}-\sum\limits_{i=0}^{\infty}d_{4i+2}-\frac{3}{2}\sum\limits_{i=0}^{\infty}d_{4i+3}-2\sum\limits_{i=1}^{\infty}d_{4i}.
\end{eqnarray*}

We may assume that for each component $B$ of $F$,  $|C_2 \cap B| \le \frac 12 |B|$, for otherwise we replace $C_2$ with the symmetric difference $C_2 \Delta B$. 
Thus $|C_2 \cap F| \le \frac 12 |F|- \frac12 \sum_{i=1}^{\infty} d_{2i+1}$.
Let $\mathcal{C}_2= \{C_1, C_2, C_1 \Delta F\}$. Then $\mathcal{C}_2$ is also a cycle cover of $G$. The length of this cycle cover is \begin{eqnarray*}
\ell(\mathcal{C}_2) &=& |C_1|+|C_2| + |C_1-F|+ |F-C_1| \\
&=& 2|C_1-F|+|C_2|+|F|  \\
&=& m +|C_1-F| + |C_2 \cap F| +|C_1 \cap C_2-F| \\
&\le& m + |C_1-F| + |C_1 \cap C_2 -F| + \frac 12 |F|- \frac 12 \sum_{i=1}^{\infty}d_{2i+1}.
\end{eqnarray*}
Let $C'_i=C_i - F$ for $i=1,2$. Then $\ell(\mathcal{C}_2) \le m +|C'_1-C'_2| + 2 |C'_1 \cap C'_2| + \frac12 |F| - \frac 12 \sum_{i=1}^{\infty}d_{2i+1}$. 

Note that $C'_i$ is a cycle in $G/F$, $C'_1-C'_2 = E_{f=(1,0)}(G-F)$, and $C'_1 \cap C'_2= E_{f=(0,1)}(G-F)$. Hence $C'_1-C'_2$ and $C'_1 \cap C'_2$ are acyclic.

 We view each of $C'_1-C'_2$ and $C'_1 \cap C'_2$ as a spanning subgraph of $G/F$. Let $c_1, c_2$ be the number of connected components of $C'_1-C'_2$ and $C'_1 \cap C'_2$, respectively. Then $|C'_1-C'_2 |= |V(G/F)| - c_1, |C'_1 \cap C'_2| = |V(G/F)| - c_2$. 

Let $d_1$ be the number of degree 2 vertices of $G$ contained in 
$G/F$. The choice of $F$ implies that $d_1 \le \frac 13 n_2$. 
A degree 2 vertex $v$ in $G/F$ is an isolated vertex in $C'_1 \cap C'_2$ if and only if the two edges $e_1,e_2$ incident to $v$ are have flow $f(e_1)=f(e_2)) \ne (0,1)$. 
Thus by the choice of $f$, at least $\frac 23 d_1$ vertices are isolated vertices in $C'_1 \cap C'_2$. Hence $c_2\ge \frac 23 d_1+1$. Similarly, $c_1+2c_2 \ge 3 \times \frac 23 d_1+3 = 2d_1+3$. 
As $|V(G/F)| = \sum_{i=1}^{\infty}d_i$, we conclude that  
\begin{eqnarray*}
   |C'_1-C'_2| + 2|C'_1 \cap C'_2|&\le&  3\sum\limits_{i=2}^{\infty}d_i+d_1-3.
 \end{eqnarray*}
 Therefore 

\begin{eqnarray*}
    \ell(\mathcal{C}_2) &\le & m+3\sum\limits_{i=2}^{\infty}d_i+d_1-3+\frac 12 |F| - \frac 12 \sum_{i=1}^{\infty} d_{2i+1}\\
&=&m+\frac{1}{2}|F|+\frac{5}{2}\sum\limits_{i=1}^{\infty}d_{4i+1}+3\sum\limits_{i=0}^{\infty}d_{4i+2}+\frac{5}{2}\sum\limits_{i=0}^{\infty}d_{4i+3}+3\sum\limits_{i=1}^{\infty}d_{4i}+d_1-3.
\end{eqnarray*}

Hence 
\begin{eqnarray*}
     cc(G)&\le &\min \{\ell(\mathcal{C}_1), \ell(\mathcal{C}_2)\} \\
     &\le &\frac{5}{6}\ell(\mathcal{C}_1)+\frac{1}{6}\ell(\mathcal{C}_2)\\
     &=& \frac{5}{6}(2m-\frac{1}{2}|F|-\frac{1}{2}\sum\limits_{i=1}^{\infty}d_{4i+1}-\sum\limits_{i=0}^{\infty}d_{4i+2}-\frac{3}{2}\sum\limits_{i=0}^{\infty}d_{4i+3}-2\sum\limits_{i=1}^{\infty}d_{4i})\\
     &+&\frac{1}{6}(m+\frac{1}{2}|F|+\frac{5}{2}\sum\limits_{i=1}^{\infty}d_{4i+1}+3\sum\limits_{i=0}^{\infty}d_{4i+2}+\frac{5}{2}\sum\limits_{i=0}^{\infty}d_{4i+3}+3\sum\limits_{i=1}^{\infty}d_{4i}+d_1-3)\\
     &\le& \frac{11}{6}m-\frac{1}{3}|F|+\frac{1}{6}d_1-\frac{1}{2}.
 \end{eqnarray*}
 As $|F| \ge \frac{2}{3}m$ and $d_1 \le \frac{1}{3}n_2$, we have $$cc(G)\le \frac{29}{18}m+\frac{1}{18}n_2-\frac{1}{2}.$$

It was proved by Jamshy, Raspaud and Tarsi \cite{JRT} that if $G$ admits a nowhere zero 5-flow, then $cc(G) \le \frac85 m$, and hence $cc(G)\le\frac{5}{3}m-\frac{25}{3}$, provided that $m \ge 125$.
Assume that $G$ does not admits a nowhere zero  5-flow, i.e., $G$ is a counterexample to Tutte's 5-flow conjecture. It was proved by Kochol \cite{Kochol} that every counterexample to the 5-flow conjecture has girth at least 11, and hence has at least 94 vertices. Note that a graph $G'$ admits a nowhere zero 5-flow if and only if by replacing each thread of $G'$ by a single edge, the resulting graph admits a nowhere zero 5-flow.  Thus we may assume that $G$ has at least $94$ vertices of degree at least $3$. Hence $94 \times 3 + 2n_2 \le 2m$ and $n_2 \le m - 141$, provided that $m \ge 141$.  This implies that if $m \ge 141$,
then $cc(G)\le\frac{5}{3}m-\frac{25}{3}$.

 It remains to show that $cc(G)\le\frac{5}{3}m - \frac 1{42} \log m$. 

As  
    $cc(G) \le \frac 53 m - \frac 1{18}(m-n_2)$, we are done if $ m-n_2 > \frac 3{7} \log m$. 

    Assume $m-n_2 \le \frac 3{7} \log m$.

Let $\mathcal{C}_1$ be the cycle cover of $G$ constructed above. We have shown that $$cc(G) \le \ell(\mathcal{C}_1) \le 2m- |F|+2|E_{f=(0,0)}(G)|,$$
where $|F| \ge \frac 23 m$.

Let $\epsilon = \frac{\log m}{56m}$. 
If $|E_{f=(0,0)}(G)| \le (\frac 14 - \epsilon)|F|$, then $$cc(G) \le 2m-(\frac 12 + 2\epsilon) |F| \le 2m-(\frac 12 + 2\epsilon) \frac 23 m = \frac 53 m - \frac 1{42}\log m.$$ 

Thus it suffices to show that for any connected component $C$ of $F$, 
$|E_{f=(0,0)}(C)| \le (\frac 14 - \epsilon)|C|$. 

If $|C| < \frac {14m}{\log m}$, then it follows from Corollary \ref{main} that $|E_{f=(0,0)}(C)| \le  \frac 14|C| - \frac14 < (\frac 14 - \epsilon)|C|$

Assume $|C|\ge \frac {14m}{\log m}$. 
Let $e_1, e_2, \ldots, e_k$ be the threads in $C$, with $\ell(e_i) \le \ell(e_{i+1})$ for all $i$. Thus $C$ has $k$ vertices of degree at least $3$. Hence $k \le 
\frac23(m-n_2) \le  \frac2{7}\log m$. 

 If $\ell(e_i) \le 5^i \log m$ for all $i$, then $ |C| \le \sum_{i=1}^k 5^i \log m \le \frac {5^{k+1}}{4} \log m$.
As $ k \le  \frac2{7}\log m$, 	$|C| \le \frac 54 m^{\frac{2log5}7} \log m$. This implies that $\frac{14 m}{\log m } \le \frac 54 m^{\frac{2\log 5}7} \log m$, hence $56m^{1-\frac{2\log5}7} \le  5(\log m)^2$, which fails for any positive integer $m$.

 Thus there exists $i$ such that $\ell(e_i) \ge 5^i \log m$.

Let $i_0$ be the minimum index such that 
$\ell(e_{i_0}) \ge 5^{i_0} \log m$. 

Let $S$ be the union of all the threads $e_1, e_2, \ldots, e_{i_0-1}$. Let $G'=G/S$ and $C' = C/S$. 
Let $q= 5^{i_0} \log m$. 
Then  each thread of $C'$ has length at least $q$.  By Corollary \ref{main}, we may assume that $|E_{f=(0,0)}(C')| \le \frac14 (|E(C')|-q)$.

Then $|E_{f=(0,0)}(C)| \le \frac 14 (|E(C')| - q) + \sum_{i=1}^{i_0-1}5^i \log m \le \frac 14 |E(C)| - \frac {5}{4} \log m< (\frac 14 - \epsilon)|C|$. This completes the proof of Theorem \ref{cc1}.
\end{proof}

\section*{Acknowledgements}

X. Zhu is supported by grant: NSFC 12371359. 

\end{document}